\documentclass[a4paper,oneside,11pt]{amsart}
\usepackage{hyperref}

 \newtheorem{thm}{Theorem}[section]

 \theoremstyle{definition}
 \newtheorem{defn}[thm]{Definition}
 \theoremstyle{remark}
 \newtheorem{rem}[thm]{Remark}
 
 \numberwithin{equation}{section}

\begin{document}

\title[Odd Unitary Analogue of Gram-Schmidt Process]
 {On the Odd Unitary Analogue of Gram-Schmidt Process}

\author[A.A. Ambily]{A.A. Ambily}

\address{
Department of Mathematics\\
Cochin University of Science and Technology\\
Kerala, India}
\email{ambily@cusat.ac.in}

\author{V.K. Aparna Pradeep}
\address{Kerala School of Mathematics\\ 
Kerala, India}
\email{aparnapradeepvk@gmail.com}



\begin{abstract}
In 1976,  L.N. Vaserstein used a construction analogous to the Gram–Schmidt orthogonalisation, for obtaining a set of symplectic matrices from a set of elementary matrices. We have a similar construction for Petrov's odd unitary group. Here, we prove that the elementary matrices in the odd unitary analogue of the Gram-Schmidt process form a set of generators for the elementary linear group.
\end{abstract}

\maketitle

\section{Introduction}
The Gram-Schmidt process introduced by J. P. Gram and E. Schmidt enables one to orthonormalize a set of vectors in an inner product space. Using this, one can transform a square matrix into an orthogonal matrix by an upper triangular matrix. This motivated L.N. Vaserstein in 1976 to introduce a similar method for symplectic matrices. In \cite{VasersteinSuslin1976}, Vaserstein proved that given an invertible alternating matrix $\varphi$ of size $2n$, there exist two elementary matrices of size $2n-1$, which can be transformed to elementary symplectic matrices with respect to $\varphi$.

\vspace{2mm}

Later, P. Chattopadhyay and R.A. Rao used Vaserstein's construction for studying certain results on the symplectic group with respect to a given invertible alternating matrix in \cite{ChattopadhyayRao2016}. They used the construction for defining the relative elementary symplectic group. In \cite{ChattopadhyayRao2024}, the authors found a new set of generators for the odd sized elementary linear group using the Vaserstein's construction. For an invertible alternating matrix $\varphi$ of size $2n$, they collected all the odd-sized elementary matrices in Vaserstein's construction and denoted the subgroup of the elementary group ${\rm E}_{2n-1}(R)$ generated by these matrices as ${\rm E}_{\varphi}(R)$. The following is the main result proved in \cite{ChattopadhyayRao2024}.
\begin{thm}{\rm \cite[Theorem 5.1]{ChattopadhyayRao2024}}\label{GSChattopadhyayRao}
Let $\varphi$ be an alternating matrix of Pfaffian 1 of size $2n$, with $n \geq2$. Then the groups  ${\rm E}_{2n-1}(R)$ and ${\rm E}_{\varphi}(R)$ are equal.
\end{thm}

Here we establish a similar result using the odd unitary analogue of Vaserstein's Gram-Schmidt process defined in \cite{AAAP2022}.

\vspace{2mm}

The odd unitary group is  a new type of classical-like group introduced by V.A. Petrov in \cite{Petrov2005}. This group is a generalisation of Bak's unitary group (\cite{Bak1969}) and thus generalises all the classical groups. It also serves as a generalisation of G. Tang's Hermitian groups (\cite{Tang1998}) and the groups ${\rm U}_{2n+1}(R)$ defined by E. Abe (see \cite{Abe1977}). This group is significant because it contains the odd dimensional orthogonal group ${\rm O}_{2n+1}$ and the groups ${\rm U}_{2n+1}(R)$ of E. Abe, which were not discussed in the theory of quadratic or Hermitian groups. Using the concept of the odd unitary group, we can provide uniform proofs for many classical results in literature. Petrov's odd unitary group and its properties were well studied in the last decade. The nilpotency of odd unitary $K_1$-functor for this group was studied by W. Yu and G. Tang in \cite{YuTang2016}. Also, W. Yu, Y. Li, and H. Liu (see \cite{YuLiLiu2018}), and A. Bak and R. Preusser (see \cite{BakPreusser2018, Preusser2020}) studied the E-normal subgroups of odd unitary groups. In \cite{AAAP2023}, we did a comparison of Petrov's odd unitary group and the Dickson-Siegel-Eichler-Roy elementary orthogonal group defined by A. Roy in 1968.

\vspace{2mm}

In \cite{AAAP2022}, we constructed Vaserstein-type matrices for Petrov's odd unitary group. We defined two matrices $\theta(v)$ and $\eta(v)$ of size $n+2m-1$, which can be used to construct unitary matrices $L(v)$ and $L(v)^{*}$ of size $n+2m$. We also proved that, under certain conditions the matrices $L(v)$ and $L(v)^{*}$ generate Petrov's odd elementary hyperbolic unitary group. In that paper, we computed the matrices corresponding to $\theta(v)$ and $\eta(v)$ and used it to prove the main result. Here we use the matrix version of $\theta(v)$ and $\eta(v)$ to prove the following theorem.
\begin{thm}
The group generated by $\theta(v)$ and $\eta(v)$ in Vaserstein-type construction for the odd unitary group coincides with the elementary linear group ${\rm E}_{n+2m-1}(R)$.  
\end{thm}

In \cite{ChattopadhyayRao2024}, the authors proved Theorem \ref{GSChattopadhyayRao} in different stages. First, they proved the result for a matrix congruent to the standard alternating matrix. Using that special case, they proved the result for general alternating matrix of Pfaffian $1$. They used different tools such as local-global principle and Rao-Swan lemma for proving the result. Here we use a more general proof technique which works for any matrix $\varphi$. A similar technique can also be done for the symplectic case and it will give a direct and simpler proof for the main result in \cite{ChattopadhyayRao2024}.

\section{Preliminaries} 
Here we describe the construction of the odd unitary group and its elementary subgroup defined by V.A. Petrov in \cite{Petrov2005}.

\subsection{Odd Unitary Groups}
Let $R$ denotes an associative ring with unity.
\begin{defn}
An additive map $\sigma:R \to R$ defined by $r \mapsto \bar{r}$ satisfying the properties $\overline{r_{1}r_{2}}=\bar{r}_{2} \bar{1}^{-1} \bar{r}_{1}$ and $\bar{\bar{r}}=r$ for all $r,r_{1}, r_{2} \in R$ is said to be a pseudo-involution on $R$.
\end{defn}

Consider a ring $R$ together with a pseudo-involution. Let $V$ be a right $R$-module. A map $\langle , \rangle:V \times V \to R$ is called a sesquilinear form on $V$ if it is biadditive and satisfy the equation
$$\langle v_{1}r,v_{2}s \rangle =\bar{r}\bar{1}^{-1}\langle v_{1},v_{2} \rangle s,$$ for all $v_{1},v_{2} \in V$ and $r,s \in R$.

\vspace{2mm}
A sesquilinear form which also satisfies the condition  $ \langle v_1,v_2 \rangle =-\overline{\langle v_2,v_1\rangle}$ for all $v_1,v_2 \in V$, is called an anti-Hermitian form on $V$.

\begin{defn}
Consider a right $R$-module $V$ together with an anti-Hermitian form $\langle, \rangle$ on it. One can easily see that the set $\mathfrak H=V \times R$ with the composition defined by
	$$(v_1,r)\dot{+}(v_2,s)=(v_1+v_2,r+s+\langle v_1,v_2 \rangle)$$ 
is a group with the identity element $(0,0)$, and the inverse of $(v,r)$ is $(-v,-r+\langle v,v\rangle)$ which we denote by $\dot{-}(v,r)$. This group $(\mathfrak H,\dot{+})$ is called the {\it Heisenberg group} associated with the form $\langle, \rangle$.
\end{defn}

The ring $R$ has a right action on $\mathfrak H$ given by $(v,r) \leftharpoonup s=(vs, \bar{s}\bar{1}^{-1}rs).$
Also, define the {\it{trace}} map on $\mathfrak H,~ tr:\mathfrak H \to R$, by   
$tr((v,r))=r-\bar{r}-\langle v,v \rangle$, which is a group homomorphism.

\vspace{2mm}
Define two subsets $\mathfrak L_{\min}$ and $\mathfrak L_{\max}$ of $\mathfrak H$ as follows:
$$\mathfrak L_{\min}=\{(0,r+\bar{r}): r \in R\} ,\hspace{2mm} \mathfrak L_{\max}=\{\zeta \in \mathfrak H: tr(\zeta)=0\}.$$

We can see that the $\mathfrak L_{\min}$ and $\mathfrak L_{\max}$ are subgroups of $\mathfrak H$ which are stable under the action of $R$, and also satisfies the condition $\mathfrak L_{\min} \leq \mathfrak L_{\max}$.

\begin{defn}
	An \emph{odd form parameter} is a subgroup $\mathfrak L$ of $\mathfrak H$ satisfying $\mathfrak L_{\min} \leq \mathfrak L \leq  \mathfrak L_{\max}$ and is stable under the action of $R$.
\end{defn}

If $\mathfrak L$ is an odd form parameter corresponding to the sesquilinear form $\langle, \rangle$, then the pair $q=(\langle, \rangle, \mathfrak L)$ is called an {\it odd quadratic form} and the pair $(V,q)$ is said to be an {\it odd quadratic space}.

\vspace{2mm}
Define the even part of an odd form parameter $\mathfrak{L}$ as $\mathfrak{L}_{ev}=\{ a \in R:(0,a) \in \mathfrak L \}$ and the even part of a quadratic space $(V,q)$ as 
$$V_{ev}=\{ v \in V:(v,a) \in \mathfrak L \mbox{ for a certain $a \in R$} \}.$$ 

Now, we can see the definition of the  odd unitary group given by V .A. Petrov in \cite{Petrov2005}. He defined it as a group of isometries on $V$ which satisfies certain conditions. 

\vspace{2mm}
Consider two $R$-modules $V$ and $V'$ together with form parameters $\mathfrak L$ and $\mathfrak L'$ respectively. We say that an isometry $f:V \to V'$ preserves the form parameters if $(f(v),r) \in \mathfrak L'$ for all $(v,r) \in \mathfrak L$. Two such isometries $f$ and $g$ are equivalent modulo $\mathfrak L'$ if  it satisfies the condition 
$$(f(v)-g(v),\langle g(v)-f(v), g(v)\rangle) \in \mathfrak L' \mbox{ for all } v \in V.$$ 
One can verify that this gives an equivalence relation on the set of all isometries from $V$ to $V'$.
If $f$ and $g$ are isometries which are equivalent modulo $\mathfrak L$, we denote it by $f \cong g({\rm mod} ~ \mathfrak L)$.
\begin{defn}[odd unitary group]
	For an odd quadratic space $(V,q)$, the group of all bijective isometries on $V$ that are equivalent to the identity map modulo $\mathfrak L$ is called the \emph{odd unitary group} and is denoted by ${\rm U}(V,q)$.
\end{defn}

Now, we can define the Eichler-Siegel-Dickson transvections. Let $v_1,v_2$ be two elements of an odd quadratic space $V$, and $r$ an element of $R$ satisfying $\langle v_1,v_2\rangle=0,~ (v_1,0) \in \mathfrak L$ and $(v_2,r) \in \mathfrak L$. Then the  transformations on $V$ of the form $T_{v_1,v_2}(r)$ defined by
$$T_{v_1,v_2}(r)(w)=w+v_1\bar{1}^{-1}(\langle v_2,w \rangle +r\langle v_1,w \rangle )+v_2\langle v_1,w \rangle,$$ are called {\it Eichler-Siegel-Dickson transvections}.

\vspace{2mm}
In \cite{Petrov2005}, Petrov proved that the transvections $T_{v_1,v_2}(r)$ lie in ${\rm U}(V,q)$ and defined certain subgroups of ${\rm U}(V,q)$ called the odd elementary hyperbolic unitary group.
\vspace{2mm}

A pair of  vectors $(v_1,v_2)$ such that $\langle v_1,v_2 \rangle =1$, $(v_1,0) \in \mathfrak L$, and $(v_2,0) \in \mathfrak L$ is called a hyperbolic pair. 
Let $\mathbb H$ be an odd quadratic space spanned by two vectors $e_{1}$ and $e_{-1}$ such that $\langle e_{1},e_{-1} \rangle =1$, equipped with the odd form parameter $\mathfrak L$ which is generated by $(e_{1},0)$ and $(e_{-1},0)$. Also, let $\mathbb H^{m}$ denote the orthogonal sum of $m$ copies of $\mathbb H$. 

\vspace{2mm}
Consider an odd quadratic space $V_{0}$ equipped with an odd quadratic form $q_0=(\langle, \rangle_0, \mathfrak L_0)$. The orthogonal sum $V=\mathbb H^{m} \oplus V_0$ is said to be the {\it odd hyperbolic unitary space} of rank $m$ corresponding to $\mathfrak L$ and the odd unitary group in this case is called the odd hyperbolic unitary group and is denoted by ${\rm U}_{2m}(R, \mathfrak L)$.

\begin{defn}
	For a quadratic space $(V,q)$, the Witt index is defined as the greatest number $m$ such that there exist $m$ mutually orthogonal hyperbolic pairs in $(V,q)$. It denoted by ${\rm ind(V,q)}$.
\end{defn}

Let $V$ be an odd quadratic space that has Witt index at least $m$. Then we can choose an embedding of $\mathbb H^{m}$ to $V$. Fix any such embedding. Then we have elements $\{e_{i}\}_{i=1, \ldots, m,-m, \ldots, -1}$ in $V$ satisfying $\langle e_{i},e_{j}\rangle =0$ for $i \neq -j,$ $\langle e_{i},e_{-i} \rangle =1$ for $i=1, \ldots, m$ and $(e_{i},0) \in \mathfrak L$. Take $V_{0}$ as the orthogonal complement of $(e_{1}, e_{-1}, \ldots, e_{m},e_{-m})$ in $V$ and $\langle, \rangle_{0}$ and $\mathfrak L_{0}$ as the restrictions of $\langle, \rangle$ and  $\mathfrak L$ respectively to $V_{0}$. Then $V_0$ is a quadratic space with the quadratic form $(\langle, \rangle_{0}, \mathfrak L_{0})$ and $V$ is isometric to the odd hyperbolic space $\mathbb H^{m} \oplus V_{0}$. Thus ${\rm U}(V,q)={\rm U}_{2m}(R, \mathfrak L)$.

\vspace{2mm}

Now, define the elementary subgroup ${\rm EU}_{2m}(R, \mathfrak L)$ of an odd hyperbolic unitary group ${\rm U}_{2m}(R, \mathfrak L)$ as follows.

\vspace{2mm}
The elementary transvections are transformations on $V$ of the form
\begin{align*}
&T_{i,j}(r)=T_{e_{-j},-e_{i}r\varepsilon_{j}}(0),~ j \neq \pm i, ~ r \in R,\\
&T_{i}(v,r)=T_{e_{i},v\varepsilon_{-i}}(-\bar{\varepsilon}_{-i}\bar{1}^{-1}r\varepsilon_{-i}),~ (v,r) \in \mathfrak L,
\end{align*}
where $i,j \in \{-m, \ldots -1, 1, \ldots, m\}$ and $\varepsilon_{i}= \begin{cases}
\bar{1}^{-1},~ i>0\\
-1, ~ i<0
\end{cases}$.

\vspace{2mm}

Petrov defined the odd elementary hyperbolic unitary group ${\rm EU}_{2m}(R, \mathfrak L)$ as the subgroup of ${\rm U}_{2m}(R, \mathfrak L)$ generated by all elementary transvections. 

\section{Vaserstein-type matrices for odd elementary hyperbolic unitary group}

Given a natural number $n$, let $R^n$ denote the space of $n$-tuples with entries from $R$. Consider an invertible alternating matrix $\varphi$ of size $2n$ which can be written in the form     
$\begin{pmatrix}
0 & -c\\
c^{t} & \mu
\end{pmatrix}$.
Then $\varphi^{-1}$ has the form $\begin{pmatrix}
0 & d\\
-d^{t} & \rho
\end{pmatrix}$, where $c,d$ are elements in $R^{2n-1}$. For $v \in R^{2n-1}$, L. N. Vaserstein (in \cite{VasersteinSuslin1976}, Lemma 5.4) considered the matrices 
$$\alpha =I_{2n-1}+d^t v \mu \mbox{ and } \beta = I_{2n-1} - \rho v^t c$$
and proved that the matrices 
$ \begin{pmatrix}
1 & 0\\
v^{t} & \alpha
\end{pmatrix}
\mbox{ and } \begin{pmatrix}
1 & v\\
0 & \beta
\end{pmatrix}$
are elements of ${\rm E}_{2n}(R) \cap {\rm Sp}_{\varphi}(R)$, where ${\rm Sp}_{\varphi}(R)$ is the group of all symplectic matrices with respect to the form $\varphi$. 

\vspace{2mm}
In \cite{AAAP2022}, we defined matrices analogous to this, for Petrov's odd elementary hyperbolic unitary group. We observed that the form matrix for Petrov's group can be written as the $\varPsi=\widetilde{\psi}_{m} \perp \varphi$, where $\varphi$ is the form matrix for the quadratic space $(V_0, \mathfrak L_0)$ and the matrices $\widetilde{\psi}_{r}$, $r \geq 1$ is defined recursively as follows.

\begin{equation} \label{X}
 \widetilde{\psi_{1}}=\begin{pmatrix}
0 & 1\\
-\bar{1} & 0
\end{pmatrix}, 
\mbox{ and } \widetilde{\psi}_{r}=\widetilde{\psi}_{1} \perp \widetilde{\psi}_{r-1} \mbox{ for } r >1.
\end{equation}

Then one can write the matrix $\varPsi$ in the form 
$\begin{pmatrix}
0 & c\\
-\bar{1} c^{t} & \mu
\end{pmatrix}$, in which 
$c=\begin{pmatrix}
1 & 0 & \ldots & 0
\end{pmatrix} \in R^{n+2m-1}$
and $\mu$ is the square matrix of order $(n+2m)$ given by $\mu=\begin{pmatrix}
0 & 0 & 0\\
0 & \widetilde{\psi}_{m-1} & 0\\
0 & 0 & \varphi
\end{pmatrix}$. 
Also we wrote $\varPsi^{-1}$ in the form
$\begin{pmatrix}
0 & d\\
-\bar{1} d^{t} & \rho
\end{pmatrix}$, where 
$d=\begin{pmatrix}
-\bar{1}^{-1} & 0 & \ldots & 0
\end{pmatrix} \in R^{n+2m-1}$ and 
$\rho=\begin{pmatrix}
0 & 0 & 0\\
0 & \widetilde{\psi}'_{m-1} & 0\\
0 & 0 & \varphi^{-1}
\end{pmatrix}$, in which $\widetilde{\psi}'_{m}$ is defined recursively as follows.
\begin{equation} \label{Y}
 \widetilde{\psi}'_{1}=\begin{pmatrix}
0 & -\bar{1}^{-1}\\
1 & 0
\end{pmatrix}
\mbox{ and } \widetilde{\psi}'_{r}=\widetilde{\psi}'_{1} \perp \widetilde{\psi}'_{r-1}, \mbox{  for  } r >1.
\end{equation}
\vspace{2mm}

Then we defined the Vaserstein type matrices for $v=(a_1, \ldots, a_{n+2m-1}) \in R^{n+2m-1}$, denoted by $L(v)$ and $L(v)^{*}$ as 
$$L(v)=\begin{pmatrix}
1 & 0\\
v^{t} & \theta(v)
\end{pmatrix}
\mbox{ and } L(v)^{*}=\begin{pmatrix}
1 & v\\
0 & \eta(v)
\end{pmatrix},$$       
where $\theta(v) =I_{n+2m-1}+d^t \overline{v} \mu$ and $\eta(v) = I_{n+2m-1} - \bar{1}^{-1} \rho \overline{v}^t c$.

\vspace{3mm}
We have also computed the the matrices $\theta(v)$ and $\eta(v)$ corresponding to the form $\Psi$. The matrix of $\theta(v)$ is 
$$ \setlength{\arraycolsep}{1 mm} \begin{pmatrix}
		1 & \overline{a}_3 & -\bar{1}^{-1} \overline{a}_2 & \ldots & \overline{a}_{2m-1} & -\bar{1}^{-1} \overline{a}_{2m-2} & -\bar{1}^{-1}(\overline{a}_{2m}\varphi_{11} & \ldots & -\bar{1}^{-1}(\overline{a}_{2m}\varphi_{1n}\\
        & & & & & & +\overline{a}_{2m+1} \varphi_{21} & & +\overline{a}_{2m+1} \varphi_{2n}\\ 
		& & & & & & +\ldots & & +\ldots\\ 
        & & & & & & +\overline{a}_{2m-1+n} \varphi_{n1}) & & +\overline{a}_{2m-1+n} \varphi_{nn})\\ 
		0 & 1 & 0 & \ldots & 0 & 0 & 0 & \ldots & 0\\
		0 & 0 & 1 & \ldots & 0 & 0 & 0 & \ldots & 0\\
		\vdots & \vdots & \vdots & \vdots & \vdots & \vdots & \vdots & \vdots & \vdots\\
		0 & 0 & 0 & \ldots & 0 & 0 & 0 & \ldots & 1
		\end{pmatrix},$$
and the matrix $\eta(v)$ is given by
\renewcommand*{\arraystretch}{1.3} 
$$ \eta(v)= \begin{pmatrix}
	1 & 0 & 0 & \ldots & 0\\
	\bar{1}^{-2} \overline{a}_3 & 1 & 0 & \ldots & 0\\
	-\bar{1}^{-1} \overline{a}_2 & 0 & 1 & \ldots & 0\\
	\vdots & \vdots & \vdots & \vdots & \vdots\\
	\bar{1}^{-2} \overline{a}_{2m-1} & 0 & 0 & \ldots & 0\\
	-\bar{1}^{-1} \overline{a}_{2m-2} & 0 & 0 & \ldots & 0\\
	-\bar{1}^{-1} (\overline{a}_{2m}(\varphi^{-1})_{11}+\overline{a}_{2m+1} (\varphi^{-1})_{12} & 0 & 0 & \ldots & 0\\
+\ldots+\overline{a}_{2m-1+n} (\varphi^{-1})_{1n}) &  &  &  &  \\
\vdots & \vdots & \vdots & \vdots & \vdots \\
	-\bar{1}^{-1} (\overline{a}_{2m}(\varphi^{-1})_{n1}+\overline{a}_{2m+1} (\varphi^{-1})_{n2} & 0 & 0 & \ldots & 1\\
+\ldots+\overline{a}_{2m-1+n} (\varphi^{-1})_{nn}) &  &  &  &  \\
	\end{pmatrix}.$$

Now using these matrices, we can prove that $\theta(v)$ and $\eta(v)$ generates the elementary group ${\rm E}_{n+2m-1}(R)$. 

\vspace{2mm}
\begin{thm}\label{GramSchmidt}
The group generated by $\theta(v)$ and $\eta(v)$ in Vaserstein-type construction for the odd unitary group coincides with the elementary linear group ${\rm E}_{n+2m-1}(R)$.  
\end{thm}
\begin{proof}
We have 
\begin{align*}
\theta(a_1, &\ldots,a_{n+2m-1})=\textstyle{\prod\limits_{k=1}^{m-1}}E_{1,2k}(\overline{a}_{2k+1}) \cdot \textstyle{\prod\limits_{k=1}^{m-1}}E_{1,2k+1}(-\bar{1}^{-1}\overline{a}_{2k})\\
& \cdot \textstyle{\prod\limits_{k=1}^{n}}E_{1,2m-1+k}(-\bar{1}^{-1}(\overline{a}_{2m}\varphi_{1k}+\overline{a}_{2m+1}\varphi_{2k}+ \ldots+ \overline{a}_{2m-1+n}\varphi_{nk}).
\end{align*}
Also,
\begin{align*}
&\eta(v)(a_1, \ldots,a_{n+2m-1})=\textstyle{\prod\limits_{k=1}^{m-1}}E_{2k,1}(\bar{1}^{-2}\overline{a}_{2k+1}) \cdot \textstyle{\prod\limits_{k=1}^{m-1}}E_{2k+1,1}(-\bar{1}^{-1}\overline{a}_{2k})\\
& \cdot \textstyle{\prod\limits_{k=1}^{n}}E_{2m-1+k,1}(-\bar{1}^{-1}(\overline{a}_{2m}(\varphi^{-1})_{k1}+\overline{a}_{2m+1}(\varphi^{-1})_{k2}+ \ldots+ \overline{a}_{2m-1+n}(\varphi^{-1})_{k1}).
\end{align*}
Since the RHS of the above equations belong to ${\rm E}_{n+2m-1}(R)$, we have $\theta(v)$ and $\eta(v)$ belong to ${\rm E}_{n+2m-1}(R)$.
\vspace{2mm}

Now we can prove that $\theta(v)$ and $\eta(v)$ will generate ${\rm E}_{n+2m-1}(R)$. It suffices to show that the generators $E_{1,i}(a)$ and $E_{i,1}(a)$ of ${\rm E}_{n+2m-1}(R)$ can be written as $\theta(v)$ or $\eta(v)$ for some $v \in R^{n+2m-1}$. 
\vspace{2mm}

We can see that $E_{1,i}(a)$, $a \in R$ and $2 \leq i \leq n+2m-1$ can be written in terms of $\theta(v)$ as follows.
\begin{align*}
 & E_{1,2k}(a)=\theta(\bar{a}e_{2k+1}), \mbox{ for } k \in \{1, \ldots, m-1\},\\
 & E_{1,2k+1}(a)=\theta(-\bar{1}^{-1}\bar{a}e_{2k}), \mbox{ for } k \in \{1, \ldots, m-1\},\\
 & E_{1,2m-1+k}(a)=\theta(\bar{1}^{-2}(\varphi^{-1})_{1k}\bar{a}e_{2m}+\bar{1}^{-2}(\varphi^{-1})_{2k}\bar{a}e_{2m+1}\\
 &\hspace{4cm}+ \ldots +\bar{1}^{-2}(\varphi^{-1})_{nk}\bar{a}e_{2m-1+n}), \mbox{ for } k \in \{1, \ldots, n\}.
\end{align*}
Similarly, the elements in ${\rm E}_{n+2m-1}(R)$ of the form $E_{i,1}(a)$, $a \in R$ and $2 \leq i \leq n+2m-1$ can be written in terms of $\eta(v)$.
\begin{align*}
 & E_{2k,1}(a)=\eta(\bar{1}^{-2}\bar{a}e_{2k+1}), \mbox{ for } k \in \{1, \ldots, m-1\},\\
 & E_{2k+1,1}(a)=\eta(-\bar{1}^{-1}\bar{a}e_{2k}), \mbox{ for } k \in \{1, \ldots, m-1\},\\
 & E_{2m-1+k,1}(a)=\eta(\bar{1}^{-2}\varphi_{k1}\bar{a}e_{2m}+\bar{1}^{-2}\varphi_{k2}\bar{a}e_{2m+1}\\
 &\hspace{4cm}+ \ldots +\bar{1}^{-2}\varphi_{kn}\bar{a}e_{2m-1+n}), \mbox{ for } k \in \{1, \ldots, n\}.
\end{align*}
Thus, the elementary group ${\rm E}_{n+2m-1}(R)$ is generated by $\theta(v)$ and $\eta(v)$ for $v \in R^{n+2m-1}$.
\end{proof}

\begin{rem}
 The proof of Theorem \ref{GramSchmidt} does not depend on the matrix $\varphi$ taken.
\end{rem}

\subsection*{Acknowledgment}
The first author would like to thank the Kerala State Council for Science, Technology and Environment (KSCSTE) Young Scientist Award Scheme (KSCSTE/1052/2021-KSYSA-RG) for providing grant to support this work. She would also like to thank Science and Engineering Research Board (SERB) for SURE grant (SUR/2022/004894). The second author thank the National Board of Higher Mathematics (NBHM) and KSCSTE for the postdoctoral fellowship (KSOM/131/2020-AO) at Kerala School of Mathematics.

\end{document}